\documentclass[a4paper,10pt]{amsart}
\usepackage[applemac]{inputenc}
\usepackage{amsxtra,amssymb,amsthm,amsmath,amscd,mathrsfs, epsfig, eufrak}
\usepackage{amscd, amsmath, mathrsfs, amssymb, amsthm, amsxtra, bbding, epsfig, eucal, eufrak, graphicx, latexsym, mathrsfs, mathbbol, bbold}
\usepackage[all]{xy}
\usepackage[margin=1.4in]{geometry}
\usepackage{tikz}

\usepackage[normalem]{ulem}
\usepackage{soul}
\usepackage{color}

\setstcolor{red}

\makeatletter
\@namedef{subjclassname@2010}{%
  \textup{2010} Mathematics Subject Classification}
\makeatother

\def \R {{\mathbb R}}

\def \Z {{\mathbb Z}}

\def\e{{\rm e}}
\def\ii{{\rm i}}
\def \d {\,{\rm d}}
\def\re{{\Re e\,}}

\newcommand{\ag}{{\mathfrak{a}}}
\newcommand{\bg}{{\mathfrak{b}}}
\newcommand{\cg}{{\mathfrak{c}}}
\newcommand{\Ig}{{\mathfrak I}}

\theoremstyle{plain}
\newtheorem{theorem}{Theorem}
\newtheorem{proposition}{Proposition}[section]
\newtheorem{lemma}[proposition]{Lemma}

\theoremstyle{remark}

\numberwithin{equation}{section}

\addtocounter{footnote}{1}

\begin{document}
\title[Power partitions and saddle-point method]
{Power partitions and saddle-point method*	}
\author{ G\'erald Tenenbaum, Jie Wu and Yali Li}
\thanks{* We include here some corrections with respect to the published version.}

\address{%
Institut Élie Cartan de Lorraine\\
Universit\'e de Lorraine\\\goodbreak
BP 70239\\
54506 Vand\oe uvre-l\`es-Nancy Cedex\\
France
}
\email{gerald.tenenbaum@univ-lorraine.fr}

\address{%
CNRS LAMA 8050
\\
Laboratoire d'Analyse et de Math\'ematiques Appliqu\'ees
\\\goodbreak
Universit\'e Paris-Est Cr\'eteil
\\
94010 Cr\'eteil Cedex
\\
France
}
\email{jie.wu@math.cnrs.fr}
\address{%
Yali Li
\\
School of Mathematics and Statistics
\\
Henan University
\\
Kaifeng, Henan 475004
\\\goodbreak
P. R. China
}

\date{\today}

\begin{abstract}
For $k\geqslant 1$, denote by $p_k(n)$ the number of partitions of an integer $n$ into $k$-th powers.
In this  note, we apply  the saddle-point method to provide a new proof for the well-known asymptotic expansion of $p_k(n)$.
This approach turns out to significantly simplify  those of Wright (1934), Vaughan (2015) and Gafni (2016).
\end{abstract}


\maketitle

\section{Introduction}

Let $p(n)$ denote, as usual,  the number of unrestricted partitions of an integer $n$, i.e. the number of solutions to the equation
$$
n = a_1 + a_2 + \cdots + a_d,
$$
where $d\geqslant 1$ and the $a_j$ are positive integers such that $a_1\geqslant a_2\geqslant \cdots \geqslant a_d\geqslant 1$.
In 1918, Hardy and Ramanujan \cite{HardyRamanujan1918} proved the  asymptotic formula
\begin{equation}\label{HR:1}
p(n)
\sim \frac{\exp(\pi\sqrt{2n/3})}{4\sqrt{3}\, n}
\qquad
(n\to\infty).
\end{equation}
by using  modular properties of Jacobi's $\Delta$-function.\par 
More generally, given an integer $k\geqslant 1$, let $p_k(n)$ denote the number of partitions of the integer $n$ into $k$-th powers, i.e. the number of solutions to the equation
$$
n = a_1^k + \cdots + a_d^k 
$$
where, as before, $d\geqslant 1$ and  $a_1\geqslant a_2\geqslant \cdots \geqslant a_d\geqslant 1$.
Thus, $p_1(n) = p(n)$.\par 
In 1918 too, Hardy and Ramanujan \cite{HardyRamanujan1918} stated without proof the asymptotic formula
\begin{equation}\label{HR:2}
p_k(n)\sim \frac{\bg_k\exp\big\{\cg_kn^{1/(k+1)}\big\}}{n^{(3k+1)/(2k+2)}}
\qquad
(n\to\infty),
\end{equation}
where the constants $\mathfrak{b}_k$ and $\mathfrak{c}_k$ are defined by 
\begin{align}
\ag_k
& := \{k^{-1}\zeta(1+k^{-1})\Gamma(1+k^{-1})\}^{k/(k+1)},
\label{def:ak}
\\\noalign{\vskip 1mm}
\bg_k
& := \frac{\ag_k}{(2\pi)^{(k+1)/2}\sqrt{(1+1/k)}} ,
\label{def:bk}
\\\noalign{\vskip 0,5mm}
\mathfrak{c}_k
& := (k+1)\ag_k,
\label{def:ck}
\end{align}
and $\zeta$ is the Riemann zeta-function.
In 1934,  introducing a number of complicated objects including generalised Bessel functions, Wright \cite{Wright1934} obtained an asymptotic expansion of $p_k(n)$: 
for any integer $k\geqslant 1$, there is a real sequence $\{\alpha_{kj}\}_{j\geqslant 1}$ such that, for any $J\geqslant 1$, we have
\begin{equation}\label{Wright:Asymp}
p_k(n)
=\frac{ \mathfrak{b}_k \exp(\mathfrak{c}_k(n+h_k)^{1/(k+1)})}{(n+h_k)^{(3k+1)/(2k+2)}}
\bigg\{1 + \sum_{1\leqslant j<J} \frac{(-1)^j\alpha_{kj}}{(n+h_k)^{j/(k+1)}} + O\bigg(\frac{1}{n^{J/(k+1)}}\bigg)\bigg\},
\end{equation}
where
$$
h_k := \begin{cases}
0 & \text{if $k$ is even,}
\\
(-1)^{(k+1)/2} (2\pi)^{-(k+1)} k! \zeta(k+1) & \text{if $k$ is odd,}
\end{cases}
$$
and the implied constant depends at most on $J$ and $k$. Apart from an explicit formula for $\alpha_{k1}$, no further information was given about the $\alpha_{kj}$ beyond the statement that they depend only on $k$ and $j$ and that they ``may be calculated with sufficient labour for any given values of $k$,~$j$''.\par 
Of course, taking $J=1$ in \eqref{Wright:Asymp} yields an effective form of  \eqref{HR:2}.\par 
More recently, appealing to a relatively simple implementation of the Hardy-Littlewood circle method,  Vaughan \cite{Vaughan2015} obtained an explicit version of  \eqref{Wright:Asymp} in the case $k=2$ and Gafni~\cite{Gafni2016} generalised the argument to arbitrary, fixed $k$.
\par 
Gafni states her result in the following way.
Let $X=X_k(n)$ denote the real solution to the equation
\begin{equation}\label{def:X}
n = (\ag_kX)^{1+1/k}- \tfrac{1}{2}X - \tfrac{1}{2}\zeta(-k),
\end{equation}
and write
\begin{equation}\label{def:Y}
Y=Y_k(n) := (1+1/k) \ag_k^{1+1/k} X^{1/k} - \tfrac{1}{4}\cdot
\end{equation}
Then, given any $k\geqslant 1$, there is a real sequence $\{\beta_{kj}\}_{j\geqslant 1}$ such that for any fixed $J\geqslant 1$, we have
\begin{equation}\label{Gafni:Asymp}
p_k(n)
= \frac{\exp\big\{(k+1)\ag_k^{1+1/k}X^{1/k}-\tfrac{1}{2}\big\}}{(2\pi)^{(k+1)/2}X^{3/2}Y^{1/2}}
\bigg\{1 + \sum_{1\leqslant j<J} \frac{\beta_{kj}}{Y^j} + O\bigg(\frac{1}{Y^J}\bigg)\bigg\}.
\end{equation}
It may be checked that the asymptotic formulae \eqref{Wright:Asymp} and \eqref{Gafni:Asymp} match each other.
\par 
\medskip
In this  note, our aim is to provide a new proof of \eqref{HR:2}, and indeed also of \eqref{Wright:Asymp} and \eqref{Gafni:Asymp},  by applying the saddle-point method along lines very similar to those employed in \cite{Tenenbaum} in the case of $p(n)$.\par 
Our approach appears to be significantly simpler than those of the quoted previous works. \par 
Indeed, as mentioned above,  Wright's method, which provides a sharper error term, rests on the introduction of generalised Bessel functions and so may be regarded as far more difficult. Note that Vaughan \cite{Vaughan2015} does motivate his study by the search of a `relatively simple argument'.\par 
Our method simplifies the matter further. Actually,  the saddle-point method may be seen as a crude version of the circle method in which the major arcs are reduced to a single neighbourhood of one point. In the present context, it actually provides the same accuracy.\par 
We thank the referee for pointing out to us that the Vaughan-Gafni method has been very recently generalised to enumerate other classes of partitions---see \cite{BMZ18} and \cite{DR18}. We believe that the saddle-point method could still be used in these contexts, with expected substantial simplifications in the analysis.
\par \medskip
The constants $\mathfrak{b}_k$ and $\mathfrak{c}_k$ being defined as in \eqref{def:bk} and \eqref{def:ck}, we can state the following.
\begin{theorem}\label{thm1}
Let $k\geqslant 1$ be a fixed integer. There is a real sequence $\{\gamma_{kj}\}_{j\geqslant 1}$ such that, for any given integer $J\geqslant 1$, we have
\begin{equation}\label{asymp:pkn}
p_k(n) 
= \frac{\bg_k \exp(\mathfrak{c}_kn^{1/(k+1)})}{n^{(3k+1)/(2k+2)}}
\bigg\{1 + \sum_{1\leqslant j<J} \frac{\gamma_{kj}}{n^{j/(k+1)}} + O\bigg(\frac{1}{n^{J/(k+1)}}\bigg)\bigg\}
\end{equation}
uniformly for $n\geqslant 1$.
The implied constant depends at most on $J$ and $k$.
\end{theorem}
The coefficients $\gamma_{kj}$ can be made explicit directly from the computations in our proof. For instance, we find that $\gamma_{k1}=-(11k^2+11k+2)/(24k\cg_k)$ when $k\geqslant 2$, in accordance with the expression given by Wright. (It can be checked, after some computations, that it matches Gafni's formula too.) We also have $$\gamma_{11}=-\tfrac1{48}\cg_1-1/\cg_1=-\sqrt{\frac23}\Big(\frac\pi{48}+\frac3{2\pi}\Big).$$\par 
It may be seen that $|\gamma_{kj}|$ grows like $\Gamma(j)\e^{O(j)}$ and thus that the series $\sum_{j\geqslant 1}\gamma_{kj}z^j$ has radius of convergence 0.

\section{Technical preparation}
\vskip-2mm
Define
\begin{equation}\label{def:F-Phi}
F_k(s) := \sum_{n\geqslant 0} p_k(n) \e^{-ns}\qquad (\re s>0), 
\end{equation}
so that
\begin{equation}
\label{pkn}
p_k(n)=\frac1{2\pi\ii}\int_{\sigma-\ii\pi}^{\sigma+\ii\pi}F_k(s)\e^{ns}\d s=\frac1{2\pi}\int_{-\pi}^\pi F_n(\sigma+\ii\tau)\e^{n\sigma+\ii n\tau}\d\tau.
\end{equation}
According to the principles of the saddle-point method, we aim at selecting the integration abscissa $\sigma$ as a solution $\sigma_n$ of $-F_k'(\sigma)/F_k(\sigma)=n$. 
We plainly have
\begin{equation}\label{Exp-Product:Fk(s)}
F_k(s) = \prod_{m\geqslant 1}\big(1-\e^{-m^ks}\big)^{-1}\qquad (\re s>0).
\end{equation}
 Thus, in the same half-plane, we may define a determination of $\log F_k(s)$ by the formula
 $$\Phi_k(s):=\sum_{m\geqslant 1}\log \Big(\frac1{1-\e^{-m^ks}}\Big)$$
where the  complex logarithms are taken in principal branch. Expanding throughout and inverting summations, we get 
\begin{equation}\label{Series:Phi_k(s)}
\Phi_k(s) = \sum_{m\geqslant 1}\sum_{n\geqslant 1}\frac{\e^{-m^kns}}{n}
= \sum_{r\geqslant 1} \frac{w_k(r)}{r} \e^{-rs},\quad -\Phi_k'(s)=\sum_{r\geqslant 1}w_k(r)\e^{-rs} \quad (\re s>0),
\end{equation}
where 
\vskip-6mm
$$ w_k(r):=\sum_{m^k\,\mid\, r}m^k\qquad (r\geqslant 1).
$$
Hence $-\Phi_k'(\sigma)$ decreases from $+\infty$ to $0+$ on $(0, \infty)$, and so the equation $- \Phi_k'(\sigma) = n$ has  for each integer $n\geqslant 1$ a unique real solution  $\sigma_n=\sigma_n(k)$.
Moreover, the sequence $\{\sigma_n\}_{n\geqslant 1}$ is decreasing and the trivial estimates $1\leqslant w_k(r)\leqslant r^2$ yield $1/n\ll\sigma_n\ll 1/\sqrt[3]n$.
\par
We start with an asymptotic expansion for the derivatives $\Phi_k^{(m)}(\sigma_n)$ in terms of powers of~$\sigma_n$. It turns out that all coefficients but a finite number vanish.\par 
\begin{lemma}\label{lem2.1}
Let $J\geqslant 1$, $k\geqslant 1$.
As $n\to\infty$, we have
\begin{equation}\label{eq:lem2.1_0}
\Phi_k(\sigma_n)
= \frac{k\ag_k^{1+1/k}}{ \sigma_n^{1/k}}
+ \tfrac12\log\Big(\frac{\sigma_n}{(2\pi)^k}\Big)
+\tfrac12\zeta(-k)\sigma_n
+ O\big(\sigma_n^{J}\big),
\end{equation}
Moreover, for fixed $m\geqslant 1$,
\begin{equation}\label{eq:lem2.1_1}
(-1)^m \Phi_k^{(m)}(\sigma_n)
 = \frac{\ag_k^{1+1/k}}{ \sigma_n^{m+1/k}}\prod_{1\leqslant \ell<m} \Big(\ell+\frac1k\Big)
- \frac{(m-1)!}{2\sigma_n^{m}}
-\tfrac12\delta_{1m}\zeta(-k)
+ O(\sigma_n^{J}),
\end{equation}
where $\delta_{1m}$ is Kronecker's symbol. 
\end{lemma}
\goodbreak
\begin{proof}
Considering Mellin's inversion formula
$$
\e^{-s} = \frac{1}{2\pi\text{i}} \int_{2-\text{i}\infty}^{2+\text{i}\infty} \Gamma(z) s^{-z} \d z
\qquad
(\re s>0)
$$
and the convolution identity
\begin{equation}\label{def:mathfrakg}
\sum_{r\geqslant 1}\frac{w_k(r)}{r^{1+z}}
= \zeta(z+1) \zeta(kz)
\qquad
(\re z>1/k),
\end{equation}
we derive from the series representation \eqref{Series:Phi_k(s)} the integral formula
\begin{equation}\label{Integral:Phi_k(s)}
\Phi_k(s) 
= \frac{1}{2\pi\text{i}} \int_{2-\text{i}\infty}^{2+\text{i}\infty}
\zeta(z+1) \zeta(kz) \Gamma(z) \frac{\d z}{ s^z},
\end{equation}
and in turn
\begin{equation}
\label{intfder}
(-1)^m\Phi_k^{(m)}(\sigma_n)
= \frac1{2\pi\text{i}} \int_{2-\text{i}\infty}^{2+\text{i}\infty}
\zeta(z+1) \zeta(kz) \Gamma(z+m)  \frac{\d z}{\sigma_n^{z+m}}\qquad (m\geqslant 0).
\end{equation}
\par 
Using the classical fact that $\zeta(z)$ has finite order in any vertical strip $a\leqslant \re z\leqslant b$ ($a, b\in \R$ with $a<b$), or, in other words, satisfies
$$
\zeta(x+\text{i}y)\ll_{a, b} 1+|y|^{A}
\qquad
(a\leqslant x\leqslant b, \, |y|\geqslant 1),
$$
for suitable $A=A(a,b)$, and invoking Stirling's formula in the form
$$
|\Gamma(x+\text{i}y)|
= \sqrt{2\pi} |y|^{x-1/2} \e^{-\pi |y|/2} \big\{1 + O_{a, b}\big(1/y\big)\big\}
\qquad
(a\leqslant x\leqslant b, \, |y|\geqslant 1)
$$
we may move the line of integration to $\re z=-J-m-\tfrac12$.\par 
 The  shifted integral is clearly $\ll\sigma_n^{J}$.\par 
Let us first consider the case $m=0$ in \eqref{intfder}. Then the crossed singularities are a pole of order~2 at $z=0$, and two simple poles at $z=1/k$ and  $z=-1$. Indeed, $\zeta(z+1)\zeta(kz)$ vanishes at all negative integers $\leqslant -2$, so the corresponding zeros compensate the poles of $\Gamma(z)$ at negative  integers $\leqslant -2$. \par 
The residue at $z=1/k$ is equal to
$$
k^{-1} \zeta(1+k^{-1}) \Gamma(k^{-1}) \sigma_n^{-1/k}
= k\ag_k^{1+1/k} \sigma_n^{-1/k}.
$$
The residue at $z=0$ is the coefficient of $z$ in the Taylor expansion of
\begin{align*}
z^2\zeta(z+1)\zeta(kz)\Gamma(z)\sigma_n^{-z}
& = z\zeta(z+1)\zeta(kz)\Gamma(z+1)\sigma_n^{-z}
\\
& = (1-\gamma z) \{\zeta(0)+k\zeta'(0)z\} (1+\gamma z) (1-z\log \sigma_n) + O(z^2)
\\
& = \zeta(0) + \{- \zeta(0)\log\sigma_n+k\zeta'(0)\}z + O(z^2).
\end{align*}
Since $\zeta(0)=-\frac{1}{2}$ and $\zeta'(0)=-\frac{1}{2}\log(2\pi)$,
this residue equals  $\tfrac12\log\{\sigma_n/(2\pi)^k\}$.\par 
The residue at $z=-1$ equals $\tfrac12\zeta(-k)\sigma_n$.
\par\smallskip 
This completes the proof of \eqref{eq:lem2.1_0}.
\par 
When $m= 1$, the three crossed singularities are simple poles. 
The residues at $z=1/k$, $z=0$ and $z=-1$ are respectively $(1/k)\Gamma(1+1/k)\zeta(1+1/k)\sigma_n^{-1-1/k}$, $-1/2\sigma_n$ and $-\tfrac12\zeta(-k)$.
\par \smallskip
When $m\geqslant 2$, the only  crossed singularities are two simple poles,  at \mbox{$z=1/k$} and $z=0$, with respective residues
$(1/k)\Gamma(m+1/k)\zeta(1+1/k)\sigma_n^{-m-1/k}$ and $-\tfrac12(m-1)!\sigma_n^{-m}$.
This proves \eqref{eq:lem2.1_1}.
\end{proof}

\begin{lemma}\label{lem2.2}
Let $J\geqslant 1$, $k\geqslant 1$, $m\geqslant 1$.
\par
{\rm (i)}
There is a real sequence $\{a_{kj}\}_{j\geqslant 1}$ with $a_{k1} = - k/(2\cg_k)$, $a_{k2}=k/(8\cg_k^2)$, such that 
\begin{equation}\label{eq:lem2.2_A}
\sigma_n 
= \frac{\ag_k}{n^{k/(k+1)}} 
\bigg\{1 + \sum_{1\leqslant j<J} \frac{a_{kj}}{n^{j/(k+1)}}
+ O\Big(\frac{1}{n^{J/(k+1)}}\Big)\bigg\}\qquad (n\to\infty).
\end{equation}
\par
{\rm (ii)}
There is a real sequence $\{b_{kj}\}_{j\geqslant 1}$ with $b_{k1} = - a_{k1}/k$ such that, as $n\to\infty$, we have 
\begin{equation}\label{eq:lem2.2_B}
\Phi_k(\sigma_n) 
= k \ag_k n^{1/(k+1)}
\bigg\{1 + \sum_{1\leqslant j<J} \frac{b_{kj}}{n^{j/(k+1)}}
+ O\Big(\frac{1}{n^{J/(k+1)}}\Big)\bigg\}
+ \tfrac12\log\Big(\frac{\sigma_n}{(2\pi)^k}\Big).
\end{equation}
\par
{\rm (iii)}
There is a real sequence $\{b_{kmj}\}_{j\geqslant 1}$ such that, as $n\to\infty$, we have
\begin{equation}
\label{eq:lem2.2_C}
\begin{aligned}
(-1)^m& \Phi_k^{(m)}(\sigma_n)+\tfrac12\delta_{1m}\zeta(-k)\\ 
&= \frac{n^{(mk+1)/(k+1)}}{\ag_k^{m-1}} \prod_{1\leqslant \ell<m} \Big(\ell+\frac1k\Big)
\bigg\{1 + \sum_{1\leqslant j<J} \frac{b_{kmj}}{n^{j/(k+1)}}
+ O\Big(\frac{1}{n^{J/(k+1)}}\Big)\bigg\}.\\
\end{aligned}\end{equation}
\end{lemma}

\begin{proof}
We infer from \eqref{eq:lem2.1_1} that
\begin{equation}
\label{appsign}
n
= \frac{\ag_k^{1+1/k}}{\sigma_n^{1+1/k}}
- \frac{1}{2\sigma_n}
-\tfrac1{2}\zeta(-k)+ O(\sigma_n^J).
\end{equation}
This immediately implies \eqref{eq:lem2.2_A} by Lagrange's inversion formula --- see, e.g. \cite[\S7.32]{WW1927}.
We may  obtain an explicit expression for the $a_{kj}$ from the formula
\begin{equation}
\label{akj}
\sigma_n=\frac{\ag_k}{2\pi\ii\,n^{k/(k+1)}}\oint_{|z-1|=\varrho}\frac{zG'(z)}{G(z)}\d z+O\Big(\frac1{n^{(k+J)/(k+1)}}\Big)
\end{equation}
where $\varrho$ is a fixed, small positive constant and $$G(z):=z^{-1-1/k}-1-\frac{1}{2\ag_k zn^{1/(k+1)}}-\frac{\zeta(-k)}{2n}\cdot$$ This is classically derived from Rouché's theorem and we omit the details. The values of $a_{k1}$ and $a_{k2}$ may be retrieved from the above or by formally inserting \eqref{eq:lem2.2_A} into \eqref{appsign}. 
\par 
Inserting \eqref{akj} back into \eqref{eq:lem2.1_0} and \eqref{eq:lem2.1_1} immediately yields \eqref{eq:lem2.2_B} and \eqref{eq:lem2.2_C}.
\end{proof} 
%The following statement appears in \cite[Lemma 6.3]{BT04}. Here and in the sequel, we employ  the standard notation $\|t\|:=\min_{n\in \Z}|n-t|$ $(t\in\R)$.
%\begin{lemma}\label{lem2.3}
%Let $\vartheta\in \R$, $a\in \Z$, $q\in \N$ with $(a, q)=1$, $\vartheta=a/q+\beta$, $|\beta|\leqslant  1/q^2$,
% $t\in \N$, $0\leqslant  v< q$.
%Then there exist at most six integers $r$ with $0\leqslant r<q$ such that
%\begin{equation}\label{eq:lem2.3}
%\|\vartheta(tq+r)\|\in (v/q, (v+1)/q].
%\end{equation}
%\end{lemma}
With the aim of applying Laplace's method to evaluate the integral on the right-hand side of \eqref{pkn}, we need to show that it is dominated by a small neighbourhood of the saddle-point~$\sigma_n$. The next result meets this requirement. Here and in the sequel, all constants $c_j$ $(j\geqslant 0)$ are assumed, unless otherwise stated, to depend at most upon $k$.

\begin{lemma}\label{lem2.4}
We have
\begin{equation}
\label{eq:lem2.4}
\frac{|F_k(\sigma_n+\ii\tau)|}{|F_k(\sigma_n)|}
\leqslant \begin{cases}
\e^{-c_1\tau^2\sigma_n^{-(2+1/k)}} & \text{if $|\tau|\leqslant 2\pi\sigma_n$,}
\\\noalign{\vskip 1mm}
\e^{-c_2\sigma_n^{-1/k}} & \text{if $2\pi\sigma_n<|\tau|\leqslant \pi$.}
\end{cases}
\end{equation}
\end{lemma}

\begin{proof}
 Noticing that
\begin{align*}
\big|1-\e^{-m^k(\sigma_n+\ii\tau)}\big|^2
& = \big|1-\e^{-m^k\sigma_n}\big|^2+4\e^{-m^k\sigma_n}\sin^2(\tfrac{1}{2}m^k\tau)
\\
& \geqslant \big|1-\e^{-m^k\sigma_n}\big|^2+16\e^{-m^k\sigma_n} \|m^k\tau/(2\pi)\|^2,
\end{align*}
we can write
\begin{align*}
\frac{|F_k(\sigma_n+\ii\tau)|^2}{|F_k(\sigma_n)|^2}
& \leqslant \prod_{m\geqslant 1} 
\bigg(1+\frac{16 \|m^k\tau/(2\pi)\|^2}{\e^{m^k\sigma_n}(1-\e^{-m^k\sigma_n})^2}\bigg)^{-1}
\\
& \leqslant \prod_{(4\sigma_n)^{-1/k}<m\leqslant (2\sigma_n)^{-1/k}} 
\bigg(1+\frac{16\|m^k\tau/(2\pi)\|^2}{\e^{m^k\sigma_n} (1-\e^{-m^k\sigma_n})^2}\bigg)^{-1}.
\end{align*}
Thus, there is an absolute positive constant $c_3$ such that
\begin{equation}\label{Inequality}
\frac{|F_k(\sigma_n+\ii\tau)|}{|F_k(\sigma_n)|}
\leqslant \e^{-c_3S(\tau; \sigma_n)}
\end{equation}
with
$$
S(\tau; \sigma_n)
:= \sum_{m\in I} \|m^k\tau/(2\pi)\|^2.
$$
where we have put $I:=\big](4\sigma_n)^{-1/k},(2\sigma_n)^{-1/k}\big]$.\par 
If $|\tau|\leqslant 2\pi\sigma_n$ and $m\in I$, we have $|m^k\tau/(2\pi)|\leqslant \frac{1}{2}$. Thus
\begin{equation}\label{S:UB_1}
S(\tau; \sigma_n)
= \sum_{(4\sigma_n)^{-1/k}<m\leqslant (2\sigma_n)^{-1/k}} m^{2k}\tau^2/4\pi^2
\asymp \tau^2 \sigma_n^{-(2+1/k)}.
\end{equation}
When $2\pi\sigma_n<|\tau|\leqslant 2\pi\sigma_n^{1-1/3k}$, we proceed similarly, noting that for any integer $h$ with $|\tau|/4\sigma_n< 2\pi h\leqslant |\tau|/2\sigma_n$ there are $\gg(1/\sigma_n)^{1/k}/h$ integers $m$ in $I$ such that $$\tfrac14\leqslant |m^k\tau/(2\pi)-h|\leqslant \tfrac12.$$ This yields the the required estimate 
\begin{equation}
\label{S:UB_2}
S(\tau; \sigma_n)
\gg \sigma_n^{-1/k}.
\end{equation}
\par 
When $2\pi(\sigma_n)^{1-1/3k}<|\tau|\leqslant \pi$, we  note that \eqref{S:UB_2} follows, via the Cauchy-Schwarz inequality, from the inequalities
$$\sum_{m\in I}\|\tau m^k/(2\pi)\|\geqslant \sum_{m\in I}|1-\e^{ \ii\tau m^k}|\geqslant |I|-\Big|\sum_{m\in I}\e^{ \ii\tau m^k}\Big|$$
provided we can show the modulus of the last exponential sum is, say, $\leqslant (1-c)|I|$ for some positive constant $c$ depending at most upon $k$.
Now Dirichlet's approximation lemma guarantees that there exist integers 
$a\in \Z^*$ and $q$, with $1< q\leqslant (1/\sigma_n)^{1-1/3k}$, $(a,q)=1$,
$$
|\tau/(2\pi)-a/q|\leqslant (\sigma_n)^{1-1/3k}/q.
$$
If $q\leqslant (1/\sigma_n)^{1/3k}$, we readily deduce the required estimate from \cite[Lemma 2.7 \& Theorem~4.2]{Va97}. If $ (1/\sigma_n)^{1/3k}<q\leqslant (1/\sigma_n)^{1-1/3k}$, we may apply Weyl's inequality, as stated for instance in \cite[Lemma 2.4]{Va97}, to get that the exponential sum under consideration is $\ll|I|^{1-\varepsilon_k}$ for some positive $\varepsilon_k$ depending only on $k$.
Hence,  \eqref{S:UB_2} holds in all circumstances.
\end{proof}

\section{Completion of the proof}

\begin{proposition}\label{prop3.1}
Let $k\geqslant 1$, $J\geqslant 1$. Then there is a real sequence $\{e_{kj}\}_{j\geqslant 1}$ such that for any integer $J\geqslant 1$ we have
\begin{equation}\label{eq:prop3.1}
p_k(n) = \frac{\exp(n\sigma_n+\Phi_k(\sigma_n))}{\sqrt{2\pi \Phi_k''(\sigma_n)}} 
\bigg\{1 + \sum_{2\leqslant j<J} \frac{e_{kj}}{n^{j/(k+1)}} + O\bigg(\frac{1}{n^{J/(k+1)}}\bigg)\bigg\}\qquad (n\to\infty).
\end{equation}
\end{proposition}

\begin{proof}
By \eqref{pkn}, we have
\begin{equation}\label{Beginning:prop3.1}
p_k(n)= \frac{\e^{n\sigma_n}}{2\pi} \int_{-\pi}^{\pi} \e^{\Phi_k(\sigma_n+\ii\tau)+\ii n\tau} \d \tau.
\end{equation}
From \eqref{eq:lem2.4}, we deduce that
\begin{equation}
\label{majhorscol}
\begin{aligned}
\int_{2\pi\sigma_n<|\tau|\leqslant \pi} \e^{\Phi_k(\sigma_n+\ii\tau)+\ii n\tau} \d \tau&\ll\e^{\Phi_k(\sigma_n)-c_4\sigma_n^{-1/k}}\\
\int_{\sigma_n^{1+1/3k}<|\tau|\leqslant 2\pi\sigma_n} \e^{\Phi_k(\sigma_n+\ii\tau)+\ii n\tau} \d \tau&\ll\e^{\Phi_k(\sigma_n)-c_4\sigma_n^{-1/3k}}.\\
\end{aligned}
\end{equation}
Since these bounds are exponentially small with respect to the expected main term, it only remains to estimate the contribution of the interval $\Ig:=]- \sigma_n^{1+1/(3k)}, \sigma_n^{1+1/(3k)}[$, corresponding to a small neighbourhood of the saddle-point.
\par 
In this range, we have
$$
\Phi_k(\sigma_n+\ii \tau)
= \sum_{0\leqslant m\leqslant 2J+1} \frac{\Phi_k^{(m)}(\sigma_n)}{m!} (\text{i}\tau)^m + O\bigg(\frac{\tau^{2J+2}}{\sigma_n^{1/k+2J+2}}\bigg),
$$
where the estimate for the error term follows from \eqref{eq:lem2.1_1}. The same formula ensures that
$|\Phi_k^{(m)}(\sigma_n)\tau^m|\ll 1$ for $m\geqslant 3$.
Thus for $\tau\in \Ig$, we can write
\begin{align*}
&\e^{\Phi_k(\sigma_n+\ii \tau)+\ii n\tau}\\
&\quad = \e^{\Phi_k(\sigma_n) - \frac{1}{2}\Phi_k''(\sigma_n)\tau^2}
\bigg\{ 1 + \sum_{1\leqslant \ell\leqslant 2J} \frac{1}{\ell !} 
\bigg(\sum_{3\leqslant m\leqslant 2J+1} \frac{\Phi_k^{(m)}(\sigma_n)}{m!} (\ii\tau)^m\bigg)^{\ell} 
+ O\bigg(\frac{\tau^{2J+2}}{\sigma_n^{1/k+2J+2}}\bigg)\bigg\}
\\
&\quad = \e^{\Phi_k(\sigma_n) - \frac{1}{2}\Phi_k''(\sigma_n)\tau^2}
\bigg\{1 + \sum_{1\leqslant \ell\leqslant 2J} \frac{1}{\ell !} \sum_{3\ell\leqslant m\leqslant (2J+1)\ell} \lambda_{k,\ell,m}(n) \tau^m 
+ O\bigg(\frac{\tau^{2J+2}}{\sigma_n^{1/k+2J+2}}\bigg)\bigg\},
\end{align*}
where
\begin{equation}\label{def:lambda}
\lambda_{k,\ell,m}(n) := \text{i}^{m} 
\sum_{\substack{3\leqslant m_1, \dots, m_{\ell}\leqslant 2J+1\\ m_1 + \cdots + m_{\ell} = m}} 
\prod_{1\leqslant r\leqslant\ell} \frac{\Phi_k^{(m_r)}(\sigma_n)}{ m_r!}\cdot
\end{equation}
Since the contributions from odd powers of $\tau$ vanish,
we get
\begin{equation}
\label{contrcol}
\int_{\Ig} \e^{\Phi_k(\sigma_n+\ii\tau)+\ii n\tau} \d \tau
= \e^{\Phi_k(\sigma_n)}
\bigg\{I_0 + \sum_{1\leqslant \ell\leqslant 2J} \frac{1}{\ell !} \sum_{3\ell\leqslant 2m\leqslant (2J+1)\ell} \lambda_{k,\ell,2m}(n) I_m + O(R)\bigg\},
\end{equation}
with
$$
I_m := \int_{\Ig} \e^{- \frac{1}{2}\Phi_k''(\sigma_n)\tau^2} \tau^{2m} \d \tau,
\qquad 
R:= \sigma_n^{-1/k-2J-2}
\int_{\Ig}\e^{- \frac{1}{2}\Phi_k''(\sigma_n)\tau^2} \tau^{2J+2} \d \tau.
$$
 Extending the range of integration in $I_m$ involves an exponentially small error, so we get from the classical formula for Laplace integrals
$$I_m=\frac{\sqrt{2\pi}(2m)!}{m!2^m\Phi_k''(\sigma_n)^{m+1/2}}+O\Big(\e^{-c_5n^{1/(3k+3)}}\Big),\quad R\asymp \sigma_n^{1+(J+1/2)/k}\asymp\frac1{\sqrt{\Phi_k''(\sigma_n)}n^{J/(k+1)}}\cdot$$
Inserting these estimates back into \eqref{contrcol} and expanding all arising factors $\Phi_k^{(m)}(\sigma_n)$ by~\eqref{eq:lem2.2_C}, we obtain~\eqref{eq:prop3.1}.
\end{proof}
\noindent{\it Remark.}
From \eqref{majhorscol} and \eqref{contrcol} we see that, when $k\geqslant 2$,
\begin{equation}
p_k(n)=\frac{\e^{n\sigma_n+\Phi_k(\sigma_n)}}{\sqrt{2\pi\Phi_k''(\sigma_n)}}\Big\{1-\frac{2k^2+5k+2}{24k\cg_k}\Big(\frac{\sigma_n}{\ag_k}\Big)^{1/k}+O\Big(\sigma_n^{2/k}\Big)\Big\}
\end{equation}
where, in view of \eqref{eq:lem2.1_1}, the quantity inside curly brackets may be replaced by an asymptotic series in powers of $\sigma_n^{1/k}$. Inserting \eqref{eq:lem2.1_0} and \eqref{eq:lem2.2_A} in the main term, we thus get a formula which is very close to, but simpler than  \eqref{Gafni:Asymp}, since it follows from 
\eqref{appsign} that $X$ and $1/\sigma_n$ agree to any power of $\sigma_n$.
\par 
\medskip
We are now in a position to complete the proof of Theorem \ref{thm1}.  
\par 
We infer from \eqref{eq:lem2.2_A} and \eqref{eq:lem2.2_B} that
$$
n\sigma_n+\Phi_k(\sigma_n)
= \mathfrak{c}_k n^{1/(k+1)} 
+ \sum_{1\leqslant j<J} \frac{a_{kj}^*}{n^{j/(k+1)}}
+ O\Big(\frac{1}{n^{J/(k+1)}}\Big)
+ \tfrac12\log\bigg(\frac{\sigma_n}{(2\pi)^k}\bigg)
$$
with $a_{kj}^* := \ag_k (k a_{k,j+1}+b_{k,j+1})$.
Exponentiating and expanding, we get
\begin{equation}\label{Proof:Thm1_A}
\begin{aligned}
 \exp&(n\sigma_n+\Phi(\sigma_n))
\\
& = \frac{\sqrt{\sigma_n}}{(2\pi)^{k/2}}\exp\big(\mathfrak{c}_k n^{1/(k+1)}\big)
\bigg\{1 
+ \sum_{1\leqslant \ell<J} \frac{1}{\ell !} \bigg(\sum_{1\leqslant j<J} \frac{a_{kj}^*}{n^{j/(k+1)}}\bigg)^{\ell}
+ O\Big(\frac{1}{n^{J/(k+1)}}\Big)\bigg\}
\\
& = \frac{\sqrt{\sigma_n}}{(2\pi)^{k/2}}\exp\big(\mathfrak{c}_k n^{1/(k+1)}\big)
\bigg\{1 + \sum_{1\leqslant j<J} \frac{f_{kj}}{n^{j/(k+1)}} + O\Big(\frac{1}{n^{J/(k+1)}}\Big)\bigg\}
\end{aligned}
\end{equation}
with
$$
f_{kj} := \sum_{1\leqslant \ell<J} \frac{1}{\ell !} 
\sum_{\substack{1\leqslant j_1, \dots, j_{\ell}<J\\ j_1+\cdots+j_{\ell}=j}} a_{kj_1}^* \cdots a_{kj_{\ell}}^*.
$$
\par 
It remains to insert back into \eqref{eq:prop3.1} and expand $\sqrt{\sigma_n/\Phi_k''(\sigma_n)}$ according to  \eqref{eq:lem2.2_A} and \eqref{eq:lem2.2_C} with $m=2$ to obtain the required asymptotic formula.


\bigskip\bigskip\bigskip

\begin{thebibliography}{100}
\bibitem{BMZ18} B.C. Berndt, A. Malik, and A. Zaharescu, Partitions into kth-powers of
terms in an arithmetic progression, {\it Math. Z. \bf 290} (2018), 1277--1307.
%\bibitem{BT04} R. de la Bretèche \& G. Tenenbaum, Séries trigonométriques à coefficients arithmétiques, {\it J. Anal. Math. \bf92} (2004), 1--79.
\bibitem{DR18} A. Dunn and N. Robles, Polynomial partition asymptotics, {\it J.
Math. Anal. App. \bf459} (2018), 359--384.
\bibitem{Gafni2016}
A. Gafni, 
\textit{Power partitions},
J. Number Theory {\bf 163} (2016), 19--42.

%\bibitem{HardyRamanujan1917}
%G. Hardy and S. Ramanujan,
%\textit{Asymptotic formulae for the distribution of integers of various types},
%Proc. London Math. Soc. (2) {\bf 16} (1917), 112--132.

\bibitem{HardyRamanujan1918}
G. Hardy and S. Ramanujan,
\textit{Asymptotic formulae in combinatory analysis},
Proc. London Math. Soc. (2) {\bf 17} (1918), 75--115.

\bibitem{Tenenbaum}
G. Tenenbaum,
\textit{Applications de la m\'ethode du col}, 
Cours M2 (2015/2016), M\'ethodes analytiques,
Institut Élie Cartan de Lorraine, Universit\'e de Lorraine.
\bibitem{Va97}
R. C. Vaughan,
\textit{The Hardy-Littlewood method}, Cambridge tracts in mathematics, no. 125, second edition, Cambridge Unversity Press, 1997.

\bibitem{Vaughan2015}
R. C. Vaughan,   
\textit{Squares: additive questions and partitions}, 
Int. J. Number Theory {\bf 11}, no. 5 (2015), 1367-1409.

\bibitem{WW1927} E.T. Whittaker \& G.N. Watson,
 {\it A course of modern analysis} (4-i\`eme \'ed.), Cambridge
University Press, 1927.

\bibitem{Wright1934}
E. M. Wright,   
\textit{Asymptotic partition formulae, III. Partitions into $k$-th powers}, 
Acta Math. {\bf 63} (1) (1934), 143--191.



\end{thebibliography}
\end{document}